\documentclass[11pt,english]{article}
\usepackage[T1]{fontenc}
\usepackage[latin9]{inputenc}
\usepackage{textcomp}
\usepackage{amsmath}
\usepackage{amsthm}
\usepackage{amssymb}

\makeatletter
\theoremstyle{plain}
\newtheorem{thm}{\protect\theoremname}
\theoremstyle{plain}
\newtheorem{conjecture}[thm]{\protect\conjecturename}
\theoremstyle{plain}
\newtheorem{cor}[thm]{\protect\corollaryname}
\theoremstyle{plain}
\newtheorem{prop}[thm]{\protect\propositionname}
\theoremstyle{plain}
\newtheorem{lem}[thm]{\protect\lemmaname}
\theoremstyle{remark}
\newtheorem{rem}[thm]{\protect\remarkname}
\newcommand{\lyxaddress}[1]{
	\par {\raggedright #1
	\vspace{1.4em}
	\noindent\par}
}

\usepackage{mathrsfs}\usepackage{amsthm}\usepackage{amscd}

\usepackage{hyperref}
\usepackage{graphics}
\usepackage{a4wide}
\@ifundefined{definecolor}
 {\usepackage{color}}{}

\newcounter{EQNR}
\renewcommand{\contentsname}{Contents\\
{\footnotesize\normalfont(A table of contents should normally not
be included)}}

\makeatother

\usepackage{babel}
\providecommand{\conjecturename}{Conjecture}
\providecommand{\corollaryname}{Corollary}
\providecommand{\lemmaname}{Lemma}
\providecommand{\propositionname}{Proposition}
\providecommand{\remarkname}{Remark}
\providecommand{\theoremname}{Theorem}

\begin{document}
\title{Torsion groups of subexponential growth cannot act on finite-dimensional
CAT(0)-spaces without a fixed point}
\author{Hiroyasu Izeki\footnote{The first author was supported in part by JSPS Grants-in-Aid for Scientific Research Grant Number JP20H01802.}
\,and Anders Karlsson\footnote{The second author was supported in part by the Swiss NSF grants 200020-200400 and 200021-212864, and the Swedish Research Council grant 104651320.}}
\date{July 31, 2024}
\maketitle
\begin{abstract}
We show that finitely generated groups which are Liouville and without
infinite finite-dimensional linear representations must have a global
fixed point whenever they act by isometry on a finite-dimensional
complete CAT(0)-space. This provides a partial answer to an old question
in geometric group theory and proves partly a conjecture formulated
in \cite{NOP22}. It applies in particular to Grigorchuk's groups
of intermediate growth and other branch groups as well as to simple
groups with the Liouville property such as those found by Matte Bon
and by Nekrashevych. The method of proof uses ultralimits, equivariant
harmonic maps, subharmonic functions, horofunctions and random walks.
\end{abstract}

\section{Introduction}

The existence of infinite, finitely generated, torsion groups has
a long and rich history which started with a problem posed by Burnside
around the year 1900. The discovery that such groups actually could
exist came in the 1960s by work of Golod and Shafarevich. Since then
other infinite torsion groups with additional important features appeared,
notably the papers of Novikov-Adian on infinite Burnside groups and
the Grigorchuk groups having intermediate growth, see \cite{CSD21,Os22}
for recent discussions. 

From the viewpoint of geometric group theory it is a natural question
as to whether such groups can act on CAT(0)-spaces, indeed this topic
already appeared in the foundational essay by Gromov \cite[sect. 4.5.C]{Gr87}.
By a theorem of Schur \cite{S11} (or by Selberg's lemma) it is known
that any linear finitely generated torsion group must be finite, thus
one can say that every action by a finitely generated torsion group
on the classical symmetric spaces of nonpositive curvature must fix
a point by the Cartan fixed point theorem. Another instance is the
elegant argument in Serre's book \cite[sect. 6.5]{Se80} that finitely
generated torsion groups acting on a tree must have a global fixed-point.
By the end of the 1990s it was a well-known open problem formulated
as follows (see \cite{Sw99,Be00}): 

\emph{Can a finitely generated group that acts properly discontinuously
(and cocompactly) by isometry on a proper CAT(0)-space contain an
infinite torsion group?}

At least at that time the expected answer depended on whether the
cocompact assumption is included or not. It was again referred to
as one of the key open problems in \cite{Ca14}. 

Grigorchuk constructed a continuum of finitely generated torsion groups
of growth strictly between polynomial and exponential (thus amenable)
in the early 1980s, see \cite{G80,G84,dlH00,CSD21} and references
therein. Amenable groups admit proper actions on Hilbert spaces (\cite{BCV95}).
Both Grigorchuk groups and Burnside groups, that is, infinite torsion
groups of bounded exponent, can act without a global fixed point on
infinite-dimensional CAT(0) cubical complexes (see \cite{Sa95,Os18,Sc22}).
Also note that \cite[Theorem 25]{Se80} shows that every infinitely
generated group acts on a tree without a global fixed-point, see also
\cite[II.7.11]{BH99} for another infinitely generated torsion group
example. We refer to the recent paper \cite{HO22} for a historical
survey on this problem, its variants and with references to the many
previous partial results. In \cite{NOP22,HO22} the following conjecture
is stated:
\begin{conjecture}
\label{conj:torsion}(\cite{NOP22}) Every finitely generated group
acting without a global fixed point on a finite-dimensional CAT(0)
complex contains an element of infinite order. 
\end{conjecture}

For CAT(0) cubical complexes this is known from a paper by Sageev
\cite{Sa95} (see \cite{LV20} for a discussion), and recently extended
to CAT(0) cubical complexes without infinite cubes in \cite{GLU24}.
Norin, Osajda, and Przytycki recently proved the conjecture for two-dimensional
CAT(0) complexes (with certain additional minor assumptions) as the
main result of \cite{NOP22}.  We say that a finitely generated group
is \emph{weakly Liouville} if it admits a symmetric measure $\mu$
of finite support for which the drift $\ell$ of the corresponding
random walk is $0$: 
\[
\ell=\lim_{n\rightarrow\infty}\frac{1}{n}\sum_{g\in\Gamma}\left\Vert g\right\Vert \mu^{*n}(g)=0.
\]
It implies amenability and the property gets its name because it is
known that this is equivalent to the non-existence of non-constant
bounded harmonic functions for the Laplacian defined by $\mu$, see
\cite{La23}. This class of groups contains all groups of subexponential
growth due to Avez (\cite{La23}), many groups acting on rooted trees
\cite{AOMV16} and a rich collection of groups coming from dynamical
systems \cite{M14,MNZ23}. Our main result is:
\begin{thm}
\label{thm:main}Let $\Gamma$ be a finitely generated group which
is weakly Liouville and such that every homomorphism $\Gamma\rightarrow GL_{N}(\mathbb{R})$
has finite image for any finite $N$. Whenever $\Gamma$ acts by isometries
on a complete CAT(0)-space $Y$ of finite dimension, it must have
a global fixed point in $Y$.
\end{thm}

Branch groups, which are special types of groups acting on rooted
trees, are not linear as shown by Delzant-Grigorchuk and a more general
theorem of Abért \cite{A06}, thus they have no faithful linear representation,
and at the same time many such groups are just-infinite, meaning that
every proper quotient is finite \cite{BGZ03}, thus whenever they
are Liouville, see \cite{AOMV16}, the theorem applies.

In view of Schur's theorem, Conjecture \ref{conj:torsion} is true
for any finitely generated group admitting an infinite linear representation,
so Theorem \ref{thm:main} implies:
\begin{cor}
Conjecture \ref{conj:torsion} is true for groups which are weakly
Liouville.
\end{cor}

The following special case applies in particular to Grigorchuk's torsion
groups of intermediate growth:
\begin{cor}
\label{thm:torsion}Let $\Gamma$ be a finitely generated torsion
group of subexponential growth. Whenever $\Gamma$ acts by isometries
on a complete CAT(0)-space $Y$ of finite dimension, it must have
a global fixed point in $Y$. 
\end{cor}

As pointed out above, the finite generation and finite-dimensionality
assumptions are necessary, and let us also emphasize that our result
is more general in the sense that: 
\begin{itemize}
\item the space $Y$ is not necessarily proper or a CAT(0)-complex, and
it is not assumed to have a large isometry group.
\item no assumption on the action is made, such as being non-elementary
(in the sense of assuming no fixed-point in $\partial Y$), properly
discontinuous, or being part of a cocompact action. 
\end{itemize}
It was previously shown by Caprace-Monod in \cite[Corollary E]{CM13}
that finitely generated groups of intermediate growth cannot be a
\emph{discrete} subgroup of $Isom(Y)$ under the assumption that the
space $Y$ is a proper CAT(0)-space having a \emph{cocompact} isometry
group. On the other hand, Grigorchuk groups are by construction subgroups
of the automorphism groups of rooted binary trees, thus they are subgroups
of the isometry groups of certain proper CAT(0)-spaces possessing
a global fixed-point, the root. One should also keep in mind that
Caprace-Lytchak, \cite[Theorem 1.6]{CL10}, extended theorems of Burger-Schroeder
and Adams-Ballmann to CAT(0) spaces of finite telescoping dimension
showing that amenable groups either fix a point at infinity or leave
a flat invariant. Our theorem by-passes the trouble with that there
might be a fixed point on the visual boundary at infinity. In this
connection, we can mention a paper of Papasoglu-Swenson \cite[Theorem 3.17]{PS18}
which shows the following: Assume that $G$ acts properly and cocompactly
on a proper CAT(0)-space. If $\Gamma$ is an infinite torsion subgroup
of $G$ then it cannot fix a point in the limit set of $\Gamma$.
(Note that in our setting, say assuming finite generation and subexponential
growth but no properness or cocompactness, the limit set is empty
since by Theorem \ref{thm:torsion} any orbit is bounded.) 

Groups which have a global fixed point whenever they act on simplicial
trees are said to have \emph{Serre's property FA}. This is the 1-dimensional
version of \emph{property $\mathrm{FA}_{n}$} in \cite{F09}, that
is, having a global fixed point whenever acting by isometry on $n$-dimensional
CAT(0) cell complexes. The literature on these topics is vast, and
we cannot survey this here. For property FA and the groups of the
type of most immediate relevance for our paper, we can refer to \cite{DG08,GS23}
for results and more information.

We may also formulate:
\begin{cor}
\label{thm:simple}Let $\Gamma$ be a finitely generated simple group
which is weakly Liouville. Whenever $\Gamma$ acts by isometries on
a complete CAT(0)-space $Y$ of finite dimension, it must have a global
fixed point in $Y$.
\end{cor}

This applies to the uncountable family of such groups established
by Matte Bon \cite{M14} as well as to the simple groups of intermediate
growth that Nekrashevych constructed in \cite{Ne18}. 

Our proof uses the technique of ultralimits of spaces and equivariant
harmonic maps, heavily inspired by Bourdon's paper \cite{Bo12}, combined
with substantial arguments and results from the first author's paper
\cite{Iz23} (much of what is not repeated here, making our paper
shorter than what it would be if all needed details from \cite{Iz23}
were reproduced). The use of random walks might be somewhat surprising.
We think that the present method, in particular the combination of
ultralimits, equivariant harmonic maps, subharmonic functions on the
group, horofunctions and random walks, could have further applicability
in geometric group theory. 

\paragraph*{Acknowledgements: }

We thank Rostislav Grigorchuk, Nicolás Matte Bon, and Tatiana Nagnibeda
for useful remarks. We are also grateful to the Geometry, Groups,
and Dynamics seminar of the ENS Lyon for helpful and encouraging comments.

\section{Ultralimits of metric spaces}

For a detailed exposition on this topic we refer to \cite[Chapter 11]{CSD21}
or \cite{Bo12,Iz23,Ly05}, and for basics on CAT(0)-spaces we refer
to \cite{BH99}. Let $\omega$ denote a non-principal ultrafilter
on $\mathbb{N}$. Let $Y_{j}$ be metric spaces and $p_{j}$ a point
in each $Y_{j}.$ The associated ultralimit metric space $(Y_{\infty},d)$
is defined as follows. First, let $Y^{\infty}$ be all (one-sided)
sequences $y_{j}\in Y_{j}$ such that $d(y_{j},p_{j})$ is bounded
in $j$. Via the triangle inequality, one can therefore define the
following function on $Y^{\infty}$,
\[
d((x_{j}),(y_{j})):=\omega-\lim d(x_{j},y_{j}).
\]

This function is clearly symmetric and satisfies the triangle inequality.
Now we identify all sequences for which $d((x_{j}),(y_{j}))=0.$ The
set of equivalence classes, denoted by $Y_{\infty}$, is thus a metric
space with the same $d$ well-defined on the quotient. It is rather
immediate that if each $Y_{j}$ is a complete CAT(0)-space, then so
is $(Y_{\infty},d)$.

The dimension of a CAT(0)-space $Y$ is defined as follows (see \cite{Kl99}
for various equivalent definitions of dimension for CAT(0)-spaces).
Given a set of points $x_{0},x_{1},...,x_{n}$ in $Y$, one has a
barycentric simplex $\sigma:\Delta_{n}\rightarrow Y$ defined by mapping
the point $(\alpha_{0},...,\alpha_{n})$ in the standard simplex to
the unique minimizer (the \emph{barycenter}), which exists by standard
CAT(0)-space theory, see \cite[Ch. II.2]{BH99}, of the uniformly
convex function 
\[
f(y)=\sum\alpha_{i}d(x_{i},y)^{2}.
\]

The dimension of $Y$ is the greatest number $D$ such that for some
points $x_{0},x_{1},...,x_{D}$ the corresponding barycentric simplex
is non-degenerate. A simplex $\sigma$ is degenerate if $\sigma(\Delta)=\sigma(\partial\Delta)$,
see \cite[Definition 4.7]{Kl99}. We repeat the argument by Lytchak
in \cite[Lemma 11.1]{Ly05} (alternatively see \cite[Proposition 6.10]{Iz23})
to show:
\begin{prop}
\label{prop:dimensions}If the dimensions of $Y_{j}$ is at most $D<\infty,$
then the dimension of $Y_{\infty}$ is at most $D$.
\end{prop}

\begin{proof}
Let $x_{0},x_{1},...,x_{n}$ be points in $Y_{\infty}$ such that
the associated barycentric simplex $\sigma$ is not degenerate. Represent
these classes by actual sequences of points in $Y$, denoted $z_{i,j}$,
here $j\rightarrow\infty$. By the continuity of barycenters in CAT(0)-spaces
and the definition of barycentric simplices implies that $\sigma$
is the ultralimit of the simplices defined by $z_{i,j}$ in $Y_{j}$.
If all of these simplices would be degenerate then so would $\sigma$.
\end{proof}

\section{Existence of equivariant harmonic maps with positive energy when
passing to certain ultralimits}

Let $\Gamma$ be a group generated by a finite set $S$. Assume that
$\Gamma$ acts by isometry on a complete metric space $Y$, that is,
there is a homomorphism $\rho:\Gamma\rightarrow Isom(Y)$. A \emph{($\rho$-)equivariant
map }is a map $f:\Gamma\rightarrow Y$ such that 
\[
f(g)=\rho(g)f(e)
\]
for all $g\in\Gamma.$ These maps are determined by $f(e)$ and hence
are just the orbit maps, $g\mapsto\rho(g)f(e)$. 

Let $\mu$ be a symmetric probability measure on $\Gamma$ whose support contains $S$. We further assume that $\mu$ has finite second moment with respect to $\rho$: \[  \sum_{\Gamma} d(f(e),f(g))^2 \mu(g)<\infty.  \] This property does not depend on the choice of an equivariant map $f$.  Consider the following \emph{($\mu$-)energy} of maps $f:\Gamma\rightarrow Y$,
\[
E(f)=\sum_{\Gamma}d(f(e),f(g))^{2}\mu(g),
\]
which clearly is finite by the finite second moment assumption. We
are trying to find a map with minimal energy among the maps which
are equivariant. Such maps are called \emph{harmonic. }Intuitively
we want to place our orbit in a minimal position, which is a generalization
of the consideration of the infimal displacement for single isometries. 

When $Y$ is a complete CAT(0)-spaces it is well-known, and already
used in the previous section about barycenters, that a convex function
such as 
\[
x\mapsto\sum_{\Gamma}d(x,\rho(g)x)^{2}\mu(g),
\]
which is the energy of the map defined by $f(e)=x$, either has a
minimum in $Y$ or a minimizing sequence $p_{j}$ that escapes to
infinity, that is $d(x,p_{j})\rightarrow\infty$ for any $x\in Y$. 

In order to deal with the second scenario, we assume that $\Gamma$
acts on a finite-dimensional CAT(0)-space $Y$ without a global fixed
point. We now follow Bourdon's techniques for $L^{p}$-spaces in \cite{Bo12}.
Let

\[
\delta(x)=\max_{s\in S}d(x,s(x)),
\]
where we from now on suppress $\rho$ in the notation. Later on, we
want $\inf_{x\in Y}\delta(x)>0$, which clearly would imply that 
\[
\inf E(f)\geq\inf_{x\in Y}\delta(x)^{2}\cdot\min_{s\in S}\mu(s)>0
\]
over all equivariant maps $f$, see below. Since $S$ is a generating
set, the action has a global fixed point in $Y$ if and only if there
exists an equivariant map $f$ such that $E(f)=0$.

Suppose instead that $\inf_{x\in Y}\delta(x)=0$. We will now do the
initial scaling of the action using ultralimits. 

First, we rewrite a lemma that Bourdon attributes to Shalom \cite[Lemma 6.3]{Sh00}. 
\begin{lem}
\label{lem:Shalom}Assume that $\inf_{x\in Y}\delta(x)=0$ but there
is no global fixed point. Then for any integer $n>0$, there exists
$r_{n}>0$ and $v_{n}\in Y$ such that $\delta(v_{n})\leq r_{n}/n$
and for all $v$ with $d(v_{n},v)<r_{n},$ it holds that 
\[
\delta(v)\geq r_{n}/2n.
\]
\end{lem}

\begin{proof}
Let $n$ be a positive integer. Take $w_{1}\in Y$ such that $\delta(w_{1})\leq1/2n.$
If $\delta(v)\ge1/4n$ for every $v$ with $d(v,w_{1})<1/2$, we let
$v_{n}=w_{1}$ and $r_{n}=1/2$. Otherwise, there is $w_{2}$ with
$d(w_{1},w_{2})<1/2$ and $\delta(w_{2})\leq1/4n$. If $\delta(v)\geq1/8n$
for every $v$ with $d(v,w_{2})<1/4$, then we let $v_{n}=w_{2}$
and $r_{n}=1/4$. This procedure terminates, because otherwise we
would get a sequence $w_{k}$ such that $d(w_{k},w_{k+1})<2^{-k}$.
By completeness of $Y$ this Cauchy sequence converges to a point
$w\in Y$. But note that 
\[
\left|d(w,s(w))-d(w_{k},s(w_{k}))\right|\leq2d(w,w_{k}),
\]
which implies that $\delta(w)=0$ since $\delta(w_{k})\leq1/(2^{k}n)$.
This contradicts the assumption that the orbit did not have a global
fixed point.
\end{proof}
Using this lemma we can now show, what essentially is part of \cite[Prop. 3.1]{Bo12}:
\begin{prop}
\label{prop:delta}Let $\Gamma$ be a group generated by a finite
set $S$. Assume that it acts on a finite-dimensional complete CAT(0)-space
$Y$, with 
\[
\inf_{x\in Y}\max_{s\in S}d(x,s(x))=0
\]
but there is no global fixed point. Then by passing to a certain rescaled
ultralimit action, there is a finite-dimensional complete CAT(0)-space
$Y_{\infty}$ on which $\Gamma$ acts without a global fixed point
and 
\[
\inf_{x\in Y_{\infty}}\max_{s\in S}d(x,s(x))>0.
\]
\end{prop}

\begin{proof}
Let $\lambda_{n}=n/r_{n}$ with $r_{n}$ and $v_{n}$ from Lemma \ref{lem:Shalom}
with $\delta(x)=\max_{s\in S}d(x,s(x))$ whose infimum is zero. Consider
the metric spaces $\lambda_{n}Y=(Y,\lambda_{n}d)$, which of course
again is a complete CAT(0) on which $\Gamma$ acts by isometry via
the same action. Take the ultralimit space $Y_{\infty}$ of the sequence
of pointed spaces $(\lambda_{n}Y,v_{n})$ as in the previous section.
Thanks to Proposition \ref{prop:dimensions} it is still finite-dimensional.
Let
\[
\delta_{n}(x)=\max_{s\in S}\lambda_{n}d(x,s(x)),
\]
so $\delta_{n}(v_{n})\leq1$ and $\delta_{n}(v)\geq1/2$ for all $v$
in the $\text{\ensuremath{\lambda_{n}d}}$-ball of radius $n$ around
$v_{n}$. Thanks to the condition $\delta_{n}(v_{n})\leq1$ we have
that $\Gamma$ still acts (by isometry) on the ultralimit and thanks
to the inequality $\delta_{n}(v)\geq1/2$ on larger and larger balls
around $v_{n}$ we have that this new action of $\Gamma$ on $Y_{\infty}$
satisfy $\inf_{x\in Y_{\infty}}\delta(x)>0.$ 
\end{proof}
In summary, possibly by replacing $Y$ with an ultralimit we have
that $\Gamma$ acts on a finite dimensional CAT(0)-space $Y$ such
that
\[
\delta(x)=\max_{s\in S}d(x,s(x))
\]
is bounded away from $0$, which of course in particular excludes
the existence of a global fixed-point. This implies that the energy
of any equivariant map $f$ is also bounded away from $0$, since
\[
E(f)=\sum_{\Gamma}d(f(e),f(g))^{2}\mu(g)=\sum_{\Gamma}d(f(e),gf(e))^{2}\mu(g)\geq\delta(f(e))^{2}\min_{s\in S}\mu(s).
\]

Now we pass to the other part of \cite[Prop. 3.1]{Bo12}:
\begin{prop}
\label{prop:harmonic}Let $\Gamma$ be a group generated by a finite
set $S$, and $\mu$ a symmetric probability measure with finite second
moment whose support contains $S$. Assume that it acts on a finite-dimensional
complete CAT(0)-space $Y$ without a global fixed point. Then there
is a finite-dimensional complete CAT(0)-space $Z$ on which $\Gamma$
acts by isometry, such that there is a harmonic equivariant map $h:\Gamma\rightarrow Z$
with strictly positive energy.
\end{prop}

\begin{proof}
In view of Proposition \ref{prop:delta}, by passing to a certain
ultralimit if necessary, we may assume that $\delta(x)=\max d(x,s(x))$
bounded away from $0$ and hence the energy $E(f)$ of equivariant
maps $f$ is also bounded away from $0$ in view of the above displayed
inequality. 

Let $p_{j}$ be a sequence such that the equivariant maps with $f_{j}(e):=p_{j}$
have energy approaching the minimum. As remarked above in case the
the $p_{j}$ stay bounded there is an equivariant harmonic map because
of the uniform convexity property of complete CAT(0)-spaces. In any
case, we can argue as follows, consider the $\omega$-ultralimit of
$(Y,p_{j})$. By the above displayed inequality we have $\delta(p_{j})$
is bounded from above, say by twice the square root of the minimal
energy. This guarantees that the $\Gamma$ action extends to this
limit space by isometry, and as before this limit is a complete finite
dimensional CAT(0)-space (in particular in view of Proposition \ref{prop:dimensions}).
Likewise we may also define the map $h:\Gamma\rightarrow Y_{\infty}$
by $h(g)=\omega-\lim f_{j}(g)$. Since for each $j$ this is equivariant
so is $g$. We now claim it is moreover harmonic. Clearly, $E(h)=\omega-\lim E(f_{j})=\inf_{f}E(f).$
To show that this value of the energy is also minimal in the limit
space, take $h_{1}$ an equivariant map $\Gamma\rightarrow Y_{\infty}$
and consider its energy. Take a sequence $w_{j}$ representing $h_{1}(e)$.
Define the corresponding equivariant maps $f_{1,j}(e)=w_{j}.$ For
a given $g\in\Gamma$ let $z_{j}$ be a sequence representing $g(w_{j})$,
by equivariance we now have
\[
0=\omega-\lim d(gw_{j},z_{j})=\omega-\lim d(f_{1,j}(g),z_{j})
\]
which shows that we can represent $h_{1}$ by the maps $f_{1,j}$,
but then it follows that 
\[
E(h_{1})=\omega-\lim E(f_{1,j})\geq\inf_{f}E(f)=E(h)
\]
.
\end{proof}

\section{Proofs of the main theorems}

Let $\Gamma$ be a weakly Liouville group. Then, by definition, $\Gamma$
admits a symmetric probability measure of finite support $S$ that
generates $\Gamma$ and having zero drift. For any isometric action
$\rho:\Gamma\rightarrow Isom(X,d)$, the orbit map $g\mapsto\rho(g)x$
from the group with its word metric to the metric space $(X,d)$ is
Lipschitz. Taken together this implies that 

\[
\ell=\lim_{n\rightarrow\infty}\frac{1}{n}\sum_{g\in\Gamma}d(x,\rho(g)x)\mu^{*n}(g)=0,
\]
where $\mu^{*n}$ are the convolution powers giving the distribution
of the random walk at time $n$. We emphasize that this holds for
any isometric action of $\Gamma$ on a metric space. 

Let $\Gamma$ act by isometry on a finite-dimensional complete CAT(0)-space
$Y$. Assume that $\Gamma$ has no global fixed point in $Y$, and
take any measure with finite second moment so that its support contains
the set $S$. Then the ultralimits, in two steps as explained in the
previous section, Propositions \ref{prop:delta} and \ref{prop:harmonic},
give another finite-dimensional complete CAT(0) space $Z$ on which
$\Gamma$ acts by isometry and with the existence of an equivariant
harmonic map with positive energy. Since the energy is strictly positive
there is no fixed point.

As is explained in the appendix we may modify the measure $\mu$ so
that we can assume that its support contains all non-identity elements
of $\Gamma$ keeping the drift $0$. The notion of horofunctions $b_{\xi}$
with respect to a base point $o$ also extends to ultralimits in natural
way (\cite[Lemma 6.6]{Iz23}). We recall the important Proposition
6.9 in \cite{Iz23} slightly tailored to fit the present notations:
\begin{prop}
\label{prop:Prop 6.9}(\cite[Prop. 6.9]{Iz23}) Let $Z$ be a complete
CAT(0)-space. Take $o\in Z$ and set $Z_{\infty}=\omega-\lim(Z,o).$
Let $\Gamma$ be a countable group equipped with a symmetric probability
measure $\mu$ whose support contains $\Gamma\setminus\{e\}.$ Suppose
that $\Gamma$ acts by isometry on $Z$ and that there exists an equivariant
harmonic map $f:\Gamma\rightarrow Z$. If the drift $\ell=0$, then
for almost all random walk trajectories $\sigma$, there exists $\xi(\sigma)\in Z_{\infty}\cup\partial Z_{\infty}$such
that a function $\varphi_{\sigma}:\Gamma\rightarrow\mathbb{R}$ defined
by \begin{equation} \varphi_{\sigma}(\xi)=\begin{cases}d(\xi(\sigma),f(g))-d(\xi(\sigma),f(e)),  &\text{if $\xi(\sigma)\in Z_{\infty}$} \\ b_{\xi(\sigma)}(f(g),f(e)), &\text{if $\xi(\sigma)\in \partial Z_{\infty}$}\end{cases}\end{equation}is
$\mu$-harmonic. 
\end{prop}

The last notion of $\mu$-harmonic means that the Laplacian defined
by $\mu$ annihilates $\varphi_{\sigma}.$ The proof requires a few
pages and since it can be used in a form directly applicable in the
present paper, we refer to \cite{Iz23} for the proof. 

Using this proposition one can show the following:
\begin{prop}
\label{prop:Prop9} Let $Z$ be a complete CAT(0)-space of finite
dimension and $\Gamma$ be a countable group equipped with a symmetric
probability measure $\mu$ whose support contains $\Gamma\setminus\{e\}.$
Suppose that $\Gamma$ acts on $Z$ by isometry and that there exists
an equivariant harmonic map $f:\Gamma\rightarrow Z$. If the drift
$\ell=0$, then there is a $\Gamma$-invariant convex subset $F$
isometric to a finite dimensional (possibly zero-dimensional) Euclidean
space.
\end{prop}

\begin{proof}
Although this assertion is not explicitly stated in \cite{Iz23},
it can be proved by following the arguments in \cite[§6.4]{Iz23}
appropriately in our setting. We sketch the outline here. 

Suppose $\xi(\sigma) \in Z_{\infty}$ in Proposition \ref{prop:Prop 6.9}.   Then, as in the first part of \cite[\S 6.4]{Iz23}, there is a point in  $Z_{\infty}$ fixed by $\Gamma$, and hence every orbit of $\Gamma$ in  $Z_{\infty}$ is bounded.  In particular, any orbit of $\Gamma$ in $Z$ is bounded, and this implies  the existence of a global fixed point, an invariant zero-dimensional  Euclidean space, in $Z$ via the Bruhat-Tits fixed point theorem, \cite[II.2.8]{BH99}.  

If $\xi(\sigma) \in \partial Z_{\infty}$ in Proposition \ref{prop:Prop  6.9}, then the horofunction   $\varphi_{\sigma}: g\mapsto b_{\xi(\sigma)}(f(g),f(e))$ is  $\mu$-harmonic.    Since horofunctions on CAT(0) spaces are convex, the  pull-back of them by a harmonic map are $\mu$-subharmonic functions on  $\Gamma$ in general (see for example \cite[Prop.~2.17]{Iz23}).   Here a function $u$ on $\Gamma$ is $\mu$-subharmonic if the Laplacian  of $u$ defined by $\mu$ is nonpositive. If $u$ is the pull-back of a  convex function on a CAT(0) space by a harmonic map, then the absolute  value of its Laplacian measures the strength of the convexity of the  function around the image of the harmonic map.  Thus the fact $\varphi_{\sigma}:g\mapsto b_{\xi(\sigma)}(f(g),f(e))$ is  $\mu$-harmonic means that the horofunction has the weakest convexity  around the image $f(\Gamma)$ of $f$. This suggests that the convex  hull of $f(\Gamma)$ is flat,  and it can be proved as follows.  

Since we assume that $Z$ is finite-dimensional, a compactness principle  \cite[Theorem 6.12]{Iz23} tells us that $\xi(\sigma)$ actually lies in  $\partial Z$.   The compactness principle and \cite[Lemma 5.1]{Iz23} imply that there exist $\xi \in \partial Z$ and a bi-infinite  geodesic $c_g$ passing through $f(g)$ and terminating at $\xi$ for each  $g\in \Gamma$, and all the $c_g$ geodesics are parallel to each other (in order to prove this,  we  use the weakest convexity of the horofunction).   Then an argument involving the splitting theorem, see \cite[Ch. II.2]{BH99}, leads us to see that  there is a splitting of the convex hull of $f(\Gamma)$ one of whose  factor is a Hilbert space, and the action of $\Gamma$ also splits along this splitting.  Taking a maximal flat factor in this splitting, and  applying the argument above to the other factor, we see that the action  of $\Gamma$ on the other factor must fix a point; the action leaves the  product $F$ of the maximal flat factor and the fixed point in the  other factor invariant.   Since $Z$ is finite-dimensional, so is $F$, and it becomes the desired  flat subset in $Z$.  
\end{proof}
Recall now that we have excluded the existence of a global fixed point
in $Z$ in the preceding section. Therefore, by Proposition \ref{prop:Prop9},
we obtain a representation of $\Gamma$ into the group of affine isometries
of a Euclidean space of finite and strictly positive dimension. Thus
we obtain a linear representation of $\Gamma$ over the real numbers.
By the assumption of Theorem \ref{thm:main}, the image of this representation is finite. Hence by the Cartan-Bruhat-Tits fixed point theorem it must fix a point in this flat. This leads to a contradiction to the construction of the space $Z$ from Proposition \ref{prop:harmonic}. The conclusion is then that there must be a global fixed point for the original action, $\rho:\Gamma\rightarrow Isom(Y)$ which concludes the proof of Theorem \ref{thm:main}.

In the case of a finitely generated torsion group, by Schur's theorem
the linear representation above must have a finite image. Hence by
the Cartan-Bruhat-Tits fixed point theorem it must fix a point in
this flat. This leads to a contradiction to the construction of the
space $Z$ from Proposition \ref{prop:harmonic}. The conclusion is
then that there must be a global fixed point for the original action,
$\rho:\Gamma\rightarrow Isom(Y)$ which concludes the proof of Theorem
\ref{thm:torsion}.

In the case $\Gamma$ is a simple group that is weakly Liouville,
then the linear representation above must either be trivial or faithful.
The former scenario directly contradicts the lack of fixed points,
and in the latter scenario we can invoke the Tits alternative theorem
and conclude that $\Gamma$ must be virtually solvable (since nonamenable
groups always have positive drift of random walks). Let $H$ be a
finite index solvable subgroup of $\Gamma,$ then by considering the
kernel of the action of $\Gamma$ on the coset space $\Gamma/H$ we
must have, by simplicity, that $H=\Gamma$. As a both solvable and
simple group, $\Gamma$ must be the trivial group or a cyclic group
of prime order, and again we have a fixed point in $Z$, leading to
the desired contradiction, which allows us to conclude that there
must be a global fixed point for the original action, $\rho:\Gamma\rightarrow Isom(Y)$.
\begin{rem}
Suppose $\ensuremath{\Gamma=\mathbb{Z}}$ then the drift is surely
$0$ for the simple symmetric random walk. This group may act without
a fixed point, precisely when being a parabolic or hyperbolic isometry
of $Y$. Note that the proof here applies all the way until the point
that it preserves a flat in $Z$ so it has an axis (and the hypothesis
of only possessing finite linear representations is obviously not
satisfied). This clearly does not contradict the no fixed-point assumption,
and is also consistent with the case when $\Gamma$ acts parabolically
(i.e. without an axis) on the original space $Y$.
\end{rem}

\section*{Appendix: the drift of random walks for convex combinations}

Let $\Gamma$ be a countable group and $d$ a left-invariant metric
on $\Gamma$. Given a probability measure $\mu$ on $\Gamma$ with
finite first moment, the \emph{drift }of the random walk defined by
$\mu$ and its convolution powers $\mu^{*n}$ is 
\[
l_{d}(\Gamma,\mu)=\lim_{n\rightarrow\infty}\frac{1}{n}\sum_{g\in\Gamma}d(e,g)\mu^{*n}(g)
\]
which exists by the well-known subadditivity of 
\[
L_{d}^{n}(\Gamma,\mu)=\sum_{g\in\Gamma}d(e,g)\mu^{*n}(g).
\]
See for example \cite[Ch. 3]{La23} for details. We will establish
the following inequality:
\begin{prop}
\label{prop:conv_comb}Consider a convex combination of $\mu^{*n}$, that is,   $\sum_{n=1}^{\infty} a_n\mu^{*n}$, where $a_n\geq 0$ and  $\sum_{n=1}^{\infty}a_n=1$. Then we have  \begin{equation*}  l_d \left(\Gamma, \sum_{n=1}^{\infty}a_n\mu^{*n} \right)  \leq \left(1+\sum_{n=1}^{\infty} n a_n\right) l_d(\Gamma,\mu).  \end{equation*} 
\end{prop}

\begin{rem}
We should mention that Forghani gave an exact expression of  $l_d(\Gamma,\sum_{n=1}^{\infty}a_n\mu^{*n})$ in   \cite[Theorem 4.5]{forghani}.   However, we do not need this exact expression.    Also the first author gave the proof of  Proposition~\ref{prop:conv_comb}  in \cite{Iz23} under certain additional assumptions; but it turns out  to be overcomplicated.   As we will see below, the proof of Proposition~\ref{prop:conv_comb}  presented here is quite elementary.  
\end{rem}

Note that Proposition~\ref{prop:conv_comb} tells us that   $l_d(\Gamma,\mu)=0$ implies  $l_{d}(\Gamma,\sum_{n=1}^{\infty}a_n\mu^{*n})=0$ as long as  $\sum_{n=1}^{\infty}na_n <\infty$.    Furthermore, assuming $\mathrm{supp}\,\mu$ generates $\Gamma$ and   taking $a_n\not= 0$ for each $n\in \mathbb{Z}_{>0}$,   we obtain such a measure with  $\mathrm{supp}\ \sum_{n=1}^{\infty}a_n\mu^{*n} \supset \Gamma\setminus  \{e\}$.   Now suppose that $\Gamma$ is weakly Liouville.    Then there exists a symmetric probability measure $\mu$ with finite  support $S$ that generates $\Gamma$, and having  $l_{d_S}(\Gamma,\mu)=0$, where $d_S$ denotes the word metric on  $\Gamma$ with respect to $S$.   Since the identity map   $\mathrm{id}\colon (\Gamma,d_S)\rightarrow (\Gamma,d)$ is  Lipschitz for any left invariant metric $d$ on $\Gamma$,   we see that $l_d(\Gamma,\mu)=0$ for any left invariant metric $d$ on  $\Gamma$.  Moreover, since the support of $\mu^{*n}$ is finite for each  $n$, $\sum_{\Gamma} d(e,g)^2 d\mu^{*n}(g)<\infty$ for each $n$.   Thus,  by taking $\{a_n\}_{n\in \mathbb{Z}_{>0}}$  suitably, we can make a measure $\sum_{n=1}^{\infty}a_n \mu^{*n}$ to  have finite second moment.    Letting $\nu$ to be $\sum_{n=1}^{\infty}a_n \mu^{*n}$, we obtain the  following:  
\begin{cor}
Let $\Gamma$ be a weakly Liouville group.  Then $\Gamma$ admits a probability measure $\nu$ satisfying  \begin{itemize} \setlength{\itemsep}{0.4cm} \setlength{\parskip}{-0.2cm}  \item[\rm{(i)}] $\nu$ is symmetric with finite second moment,   \item[\rm{(ii)}] $\mathrm{supp}\ \nu \supset \Gamma\setminus \{e\}$, 	      and   \item[\rm{(iii)}] $l_d(\Gamma,\nu)=0$ for any left invariant metric 	      $d$. \end{itemize} 
\end{cor}

\begin{proof}[Proof of Proposition~\ref{prop:conv_comb}]

 In what follows, we fix a left invariant metric $d$ on $\Gamma$,   drop $d$ from $l_d(\Gamma,\mu)$ and $L^n_d(\Gamma,\mu)$, and denote  them simply by $l(\Gamma,\mu)$ and $L^n(\Gamma,\mu)$ respectively.  

Set  \begin{equation*}  \tilde L^n := \max \{L^m(\Gamma,\mu) \mid m\leq n \}.   \end{equation*}  Then $\{\tilde L^n\}_{n\in \mathbb{Z}_{>0}}$ is a nondecreasing sequence, and  satisfies $\tilde L^n \geq L^n(\Gamma,\mu)$ for each $n\in  \mathbb{Z}_{>0}$.    Furthermore it is subadditive as shown as follows:   Let $k_n$ be the smallest positive integer satisfying  $\tilde L^{k_n}= \tilde L^n$, then, for any $m<k_n$, we have  $L^m(\Gamma,\mu)<\tilde L^{k_n}$, and hence  $\tilde L^{k_n}=\max \{L^m(\Gamma,\mu)\mid m\leq k_n\}=L^{k_n}(\Gamma,\mu)$. Thus, for $m< k_n$, we have  \begin{equation*}  \tilde L^{n} = \tilde L^{k_n} = L^{k_n}(\Gamma,\mu)   \leq L^{k_n-m}(\Gamma,\mu) + L^m(\Gamma,\mu)   \leq \tilde L^{k_n-m} + \tilde L^m  \leq \tilde L^{n-m} + \tilde L^m  \end{equation*}  by the subadditivity of $\{L^n(\Gamma,\mu)\}_{n\in \mathbb{Z}_{>0}}$, the  definition and the nondecreasing property of   $\{\tilde L^n\}_{n\in \mathbb{Z}_{>0}}$.     For $m\geq k_n$, assuming $m<n$, we see that  \begin{equation*}  \tilde L^{n} = \tilde L^{k_n} \leq \tilde L^m   \leq  \tilde L^m + \tilde L^{n-m}  \end{equation*}  by the nondecreasing property of $\{\tilde L^n \}_{n\in \mathbb{Z}_{>0}}$.   In any case, we have shown that $\{\tilde L^n\}_{n\in \mathbb{Z}_{>0}}$ is  subadditive.    In particular, we have  \begin{equation}  \label{eq:app-2}   \tilde L^{ik} \leq i \tilde L^k  \end{equation}  for any positive integers $i$ and $k$.   As already remaked above it is well-known that the subadditivity of   $\{L^n\}_{n \in \mathbb{Z}_>0}$ and   $\{L^n(\Gamma,\mu)\}_{n \in \mathbb{Z}_{>0}}$  implies the existence of the following limits:  \begin{equation}  \label{eq:limits}  \lim_{n\to \infty}\frac{\tilde L^n}{n}=  \inf_{n \in \mathbb{Z}_{>0}} \frac{\tilde L^n}{n}, \quad  \lim_{n\to \infty}\frac{L^n(\Gamma,\mu)}{n}=  \inf_{n \in \mathbb{Z}_{>0}} \frac{L^n(\Gamma,\mu)}{n}.   \end{equation}    If $\{L^n(\Gamma, \mu)\}_{n \in \mathbb{Z}_{>0}}$ is bounded, then so is  $\{\tilde L^n\}$, and the both limit above are equal to $0$.    If $\{L(\Gamma,\mu)\}_{n \in \mathbb{Z}_{>0}}$ is unbounded, then  there exists an unbounded subsequence $\{k_n\} \subset \mathbb{Z}_{>0}$ that  satisfies $\tilde L^{k_n}=L^{k_n}(\Gamma,\mu)$ for each $k_n$;   indeed, for any $k_l$ with $\tilde L^{k_l}=L^{k_l}(\Gamma,\mu)$,   by the unboundedness of $\{L^n(\Gamma,\mu)\}_{n\in \mathbb{Z}_{>0}}$,   there exists the smallest $k_m$ satisfying   $L^{k_m}(\Gamma,\mu)>L^{k_l}(\Gamma,\mu)$, and we have   $\tilde L^{k_m}=L^{k_m}(\Gamma,\mu)$.  Thus we have  \begin{equation*}  \lim_{n\to \infty}\frac{\tilde L^{k_n}}{k_n}  =\lim_{n\to \infty}\frac{L^{k_n}(\Gamma,\mu)}{k_n}.   \end{equation*}  The limits on the both sides exist, since these limits are those   of subsequences of the convergent sequences in (\ref{eq:limits}).    Hence these limits must coincide with those of the original sequences.   This means  \begin{equation}  \label{eq:limit2}  \lim_{n\to \infty} \frac{\tilde L^n}{n}  = \lim_{n\to \infty} \frac{L^n(\Gamma,\mu)}{n}=l(\Gamma,\mu).   \end{equation}
 \medskip
 Now consider the measure $\sum_{n=1}^{\infty}a_n\mu^{*n}$.   Take $k \in \mathbb{Z}_{>0}$ and fix it for a while.   Note that $\left(\sum_{n=1}^{\infty}a_n\mu^{*n}\right)^{*k}$ is again a  convex combination of $\mu^{*j}$'s; we can express as   $\left(\sum_{i=n}^{\infty}a_n\mu^{*n}\right)^{*k}=  \sum_{j=1}^{\infty}a_j'\mu^{*j}$,   where $a_j'= \sum_{n_1+\dots + n_k=j} a_{n_1}\dots a_{n_k}$.    Furthermore, we have $\sum_{j=1}^{\infty}a_j'=1$,   since $\left(\sum_{n=1}^{\infty}a_n\mu^{*n}\right)^k$ is a probability  measure.   Recalling $\sum_{n=1}^{\infty} a_n =1$, we get  \begin{equation*}  \begin{split}   \sum_{j=1}^{\infty} ja_j'   & =\sum_{n_1, \dots, n_k} (n_1 + \dots + n_k) a_{n_1}\dots a_{n_k} \\  & = \sum_{n_2,\dots, n_k}\left(      \sum_{n_1}  (n_1 + \dots + n_k) a_{n_1}\right)a_{n_2} \dots a_{nk} \\  & = \sum_{n_2,\dots, n_k} \left(\left(\sum_{n_1}n_1 a_{n_1}\right)+       n_2 +  \dots + n_k\right) a_{n_2} \dots  a_{n_k} \\  & = \sum_{i_1} n_1 a_{n_1} + \dots + \sum_{n_k}n_k a_{n_k} =    k \sum_n n a_n.   \end{split}  \end{equation*}  Now set $b_i = \sum_{j=k(i-1)+1}^{ki} a_j'$ ($i\in \mathbb{Z}_{>0}$).   Then we get   \begin{equation*}  \sum_{i=1}^{\infty}i b_i    = \sum_{j=1}^{\infty} \left\lceil{\frac{j}{k}}\right\rceil a_j'    \leq \frac{1}{k} \left(k+ \sum_{j=1}^{\infty}j  a_j' \right)    = 1+  \sum_{n=1}^{\infty} n a_n,   \end{equation*}  where $\lceil s \rceil$ denotes the smallest integer greater than or  equal to $s \in \mathbb{R}$,   and we have used $\lceil j/k \rceil \leq 1 + (j/k)$ and   $\sum_{j=1}^{\infty} a_j'=1$.   Together with (\ref{eq:app-2}) and the nondecreasing property of   $\{\tilde L^n\}_{n\in \mathbb{Z}_{>0}}$, we obtain  \begin{equation*}      \left(1+ \sum_{n=1}^{\infty} n a_n\right) \tilde L^k     \geq \sum_{i=1}^{\infty} b_i \left(i \tilde L^k \right)     \geq \sum_{i=1}^{\infty} b_i \tilde L^{ik}    \geq \sum_{j=1}^{\infty} a_j' \tilde L^j.   \end{equation*}  On the other hand, recalling the definitions of $\tilde L^n$ and  $L^n(\Gamma,\mu)$, and that   $\sum_{j=1}^{\infty}a_j'\mu^{*j}=(\sum_{n=1}^{\infty}a_n\mu^{*n})^{*k}$,   we have  \begin{equation*}  \begin{split}    \sum_{j=1}^{\infty} a_j' \tilde L^j    & \geq \sum_{j=1}^{\infty} a_j' L^j(\Gamma,\mu)      = \sum_{j=1}^{\infty} \sum_{\Gamma} d(e,g)\ a_j' \mu^{*j}(g) \\   & = \sum_{\Gamma} d(e,g) \       \left(\sum_{n=1}^{\infty}a_n\mu^{*n}(g)\right)^{*k}     = L^k\left(\Gamma,\sum_{n=1}^{\infty}a_n\mu^{*n} \right)  \end{split}  \end{equation*}  
Therefore we get  \begin{equation*}  \left(1 + \sum_{n=1}^{\infty} n a_n\right)  \frac{\tilde L^k}{k} \geq   \frac{L^k(\Gamma,\sum_{n=1}^{\infty}a_n\mu^{*n})}{k}.   \end{equation*}  Letting $k\to \infty$ and recalling (\ref{eq:limit2}) completes the  proof of Proposition~\ref{prop:conv_comb}.  

\end{proof} 

\section*{Conflict of interest statement}

On behalf of all authors, the corresponding author states that there
is no conflict of interest.

\lyxaddress{Department of Mathematics, Keio University, Kohoku-ku, Yokohama 223-8522,
Japan.}

\lyxaddress{Section de mathématiques, Université de Genève, Case postale 64,
1211 Genève, Switzerland; Mathematics department, Uppsala University,
Box 256, 751 05 Uppsala, Sweden.}

\begin{thebibliography}{AOMV16}
\bibitem[A06]{A06}Abért, Miklós, Representing graphs by the non-commuting
relation. Publ. Math. Debrecen 69 (2006), no.3, 261\textendash 269.

\bibitem[AOMV16]{AOMV16}Amir, Gideon; Angel, Omer; Matte Bon, Nicolás;
Virág, Bálint, The Liouville property for groups acting on rooted
trees. Ann. Inst. Henri Poincaré Probab. Stat.52(2016), no.4, 1763\textendash 1783.

\bibitem[BGZ03]{BGZ03}Bartholdi, Laurent; Grigorchuk, Rostislav I.;
Šuni\'{k}, Zoran, Branch groups.Handbook of algebra, Vol. 3, 989\textendash 1112.
Handb. Algebr., 3 Elsevier/North-Holland, Amsterdam, 2003

\bibitem[BCV95]{BCV95}Bekka, M. E. B.; Cherix, P.-A.; Valette, A.
Proper affine isometric actions of amenable groups.Novikov conjectures,
index theorems and rigidity, Vol. 2 (Oberwolfach, 1993), 1\textendash 4.
London Math. Soc. Lecture Note Ser., 227 Cambridge University Press,
Cambridge, 1995

\bibitem[Be00]{Be00}Bestvina, Mladen, Questions in geometric group
theory (2000), available at https://www.math.utah.edu/\textasciitilde bestvina/eprints/questions-updated.pdf.

\bibitem[Bo12]{Bo12}Bourdon, Marc Un théorème de point fixe sur les
espaces Lp. Publ. Mat. 56 (2012), no. 2, 375\textendash 392. 

\bibitem[BH99]{BH99}Bridson, Martin R.; Haefliger, André, Metric
spaces of non-positive curvature. Grundlehren Math. Wiss., 319, Springer-Verlag,
Berlin, 1999. xxii+643 pp.

\bibitem[Ca14]{Ca14}Caprace, Pierre-Emmanuel, Lectures on proper
CAT(0) spaces and their isometry groups, IAS/Park City Math. Ser.,
21 American Mathematical Society, Providence, RI, 2014, 91\textendash 125.

\bibitem[CL10]{CL10}Caprace, Pierre-Emmanuel; Lytchak, Alexander,
At infinity of finite-dimensional CAT(0) spaces. Math. Ann. 346 (2010),
no.1, 1\textendash 21.

\bibitem[CM13]{CM13}Caprace, Pierre-Emmanuel; Monod, Nicolas, Fixed
points and amenability in non-positive curvature. Math. Ann. 356 (2013),
no.4, 1303\textendash 1337.

\bibitem[CSD21]{CSD21}Ceccherini-Silberstein, Tullio; D'Adderio,
Michele, Topics in groups and geometry\textemdash growth, amenability,
and random walks. With a foreword by Efim Zelmanov Springer Monogr.
Math. Springer, Cham, 2021. xix+464 pp.

\bibitem[DG08]{DG08}Delzant, Thomas; Grigorchuk, Rostislav, Homomorphic
images of branch groups, and Serre's property (FA).(English summary)Geometry
and dynamics of groups and spaces, 353\textendash 375. Progr. Math.,
265 Birkhäuser Verlag, Basel, 2008

\bibitem[F09]{F09}Farb, Benson, Group actions and Helly's theorem.
Adv. Math.222(2009), no.5, 1574\textendash 1588 \bibitem[Fo17]{forghani} Forghani, B. Asymptotic entropy of transformed random walks, Ergod. Th. Dynam. Sys., 37 (2017), 1480--1491. 

\bibitem[GLU24]{GLU24}Genevois, Anthony; Lonjou, Anne; Urech, Christian,
Cremona groups over finite fields, Neretin groups, and non-positively
curved cube complexes. Int. Math. Res. Not. IMRN(2024), no.1, 554\textendash 596.

\bibitem[G80]{G80}Grigor\v{c}uk, R. I. On Burnside's problem on periodic
groups.(Russian) Funktsional. Anal. i Prilozhen.14(1980), no.1, 53\textendash 54.

\bibitem[G84]{G84}Grigorchuk, R. I. Degrees of growth of finitely
generated groups and the theory of invariant means.(Russian) Izv.
Akad. Nauk SSSR Ser. Mat.48(1984), no.5, 939\textendash 985. 

\bibitem[GS23]{GS23}Grigorchuk, Rostislav, Savchuk, Dmytro, Liftable
self-similar groups and scale groups, preprint, arXiv:2312.05427

\bibitem[Gr87]{Gr87}Gromov, Mikhael, Hyperbolic groups, Essays in
group theory, 1987, pp. 75\textendash 263.

\bibitem[HO22]{HO22}Haettel, Thomas; Osajda, Damian, Locally elliptic
actions, torsion groups, and nonpositively curved spaces, preprint
2022, arXiv:2110.12431v2

\bibitem[dlH00]{dlH00}de la Harpe, Pierre, Topics in geometric group
theory. Chicago Lectures in Math. University of Chicago Press, Chicago,
IL, 2000. vi+310 pp.

\bibitem[Iz23]{Iz23}Izeki, Hiroyasu, Isometric group actions with
vanishing rate of escape on CAT(0) spaces. Geom. Funct. Anal. 33 (2023),
no. 1, 170\textendash 244. 

\bibitem[Kl99]{Kl99}Kleiner, Bruce, The local structure of length
spaces with curvature bounded above. Math. Z. 231 (1999), no. 3, 409\textendash 456.

\bibitem[La23]{La23}Lalley, Steven P., Random walks on infinite groups.
Grad. Texts in Math., 297 Springer, Cham, 2023. xii+369 pp. 

\bibitem[LV20]{LV20}Leder, Nils; Varghese, Olga A note on locally
elliptic actions on cube complexes. Innov. Incidence Geom. 18 (2020),
no. 1, 1\textendash 6.

\bibitem[Ly05]{Ly05}Lytchak, A. Rigidity of spherical buildings and
joins. Geom. Funct. Anal. 15 (2005), no. 3, 720\textendash 752. 

\bibitem[M14]{M14}Matte Bon, Nicolás, Subshifts with slow complexity
and simple groups with the Liouville property. Geom. Funct. Anal.24(2014),
no.5, 1637\textendash 1659.

\bibitem[MNZ23]{MNZ23}Matte Bon, Nicolás; Nekrashevych, Volodymyr
; Zheng, Tianyi, Liouville property for groups and conformal dimension,
https://arxiv.org/abs/2305.14545

\bibitem[Ne18]{Ne18}Nekrashevych, Volodymyr, Palindromic subshifts
and simple periodic groups of intermediate growth. Ann. of Math. (2)187(2018),
no.3, 667\textendash 719.

\bibitem[NOP22]{NOP22}Norin, Sergey; Osajda, Damian; Przytycki, Piotr
Torsion groups do not act on 2-dimensional CAT(0) complexes. Duke
Math. J. 171 (2022), no. 6, 1379\textendash 1415. 

\bibitem[Os18]{Os18}Osajda, Damian, Group cubization. With an appendix
by Mikaël Pichot, Duke Math. J.167(2018), no.6, 1049\textendash 1055.

\bibitem[Sa95]{Sa95}Sageev, Michah Ends of group pairs and non-positively
curved cube complexes. Proc. London Math. Soc. (3) 71 (1995), no.
3, 585\textendash 617.

\bibitem[Os22]{Os22}Osin, D. A simple construction of infinite finitely
generated torsion groups, preprint 2022, arXiv:2211.09989 

\bibitem[PS18]{PS18}Papasoglu, Panos; Swenson, Eric, Finite cuts
and CAT(0) boundaries, preprint 2018, arXiv:1807.04086 

\bibitem[Sc22]{Sc22}Schneeberger, Grégoire, Proper actions of Grigorchuk
groups on a CAT(0) cube complex, preprint 2022, arXiv:2207.04980v1

\bibitem[S11]{S11}Schur, I. {\"U}ber Gruppen periodischer Substitutionen,
Sitzungsb. Preuss. Akad. Wiss. (1911), 619\textendash 627.

\bibitem[Se80]{Se80} Serre, Jean-Pierre, Trees. Translated from the
French by John Stillwell. Springer-Verlag, Berlin-New York, 1980.
ix+142 pp.

\bibitem[Sh00]{Sh00}Shalom, Yehuda, Rigidity of commensurators and
irreducible lattices. Invent. Math.141(2000), no.1, 1\textendash 54. 

\bibitem[Sw99]{Sw99}Swenson, E. A cut point theorem for CAT(0) groups,
J. Differential Geom. 53:2 (1999), 327-358.

\end{thebibliography}
\end{document}